\documentclass{amsart}
\usepackage[utf8]{inputenc}
\usepackage{amsmath}
\usepackage{hyperref}
\usepackage{dynkin-diagrams}
\usepackage[shortlabels]{enumitem}

\newtheorem{theorem}{Theorem}[section]
\newtheorem{lemma}[theorem]{Lemma}

\theoremstyle{definition}
\newtheorem{definition}[theorem]{Definition}
\newtheorem{example}[theorem]{Example}

\newtheorem{corrollary}[theorem]{Corrollary}

\theoremstyle{remark}
\newtheorem{remark}[theorem]{Remark}

\numberwithin{equation}{section}

\begin{document}

\title{Octonion Integers and Tight 5-Designs}

\author{Benjamin Nasmith}
\curraddr{Department of Mathematics and Computer Science, Royal Military College of Canada, Kingston, ON}
\email{ben.nasmith@gmail.com}

\date{\today}



\begin{abstract}
    The two strictly projective tight 5-designs are the lines spanned by the short vectors of the Leech lattice and a set of points in the octonion projective plane that define a generalized hexagon of order (2,8). A previous paper introduced a common construction that can generate these two tight 5-designs. This paper describes the same construction in terms of octonion arithmetic. An octonion integer construction of the Leech lattice is described using properties of the octonion integers taken modulo 2. The Leech lattice automorphism group is constructed from octonion reflections. The common construction and the Suzuki subgroup chain of Leech lattice automorphisms are described in terms of octonion integers.
\end{abstract}

\maketitle

\section{Introduction}

Let $\mathbb{F}$ be a real division composition algebra---one of $\mathbb{R}, \mathbb{C}, \mathbb{H}, \mathbb{O}$.
A previous paper \cite{nasmith_octonions_2022} introduced a common construction for $t$-designs in $\mathbb{RP}^{\rho d -1}$ and $\mathbb{FP}^{\rho-1}$ and showed how this common construction can yield the only two strictly projective tight $5$-designs, which exist in $\mathbb{RP}^{23}$ and $\mathbb{OP}^2$.
The example used in that paper was obtained by computation and given to establish the existence of the common construction of these two tight $5$-designs.
That example matches the coordinates used in Wilson's octonion Leech lattice construction \cite{wilson_octonions_2009, wilson_conways_2011}.
This paper instead examines the common construction of these two tight $5$-designs using octonion integers and their symmetries.

We begin with a review of \cite{nasmith_octonions_2022}. 
In what follows, let rank $\rho$ and degree $d$ be positive integers with $\rho$ greater than $1$ and  $d = [\mathbb{F}:\mathbb{R}]$ when $\rho$ is greater than $2$.
We will call vector $x = (x_1, x_2, \ldots, x_\rho)$ in $\mathbb{F}^\rho$ respectively \textit{commutative} or \textit{associative} when the coefficients $x_1, x_2, \ldots, x_\rho$ belong to a common commutative or associative subalgebra of $\mathbb{F}$.
The norm $N(x)$ is defined as $N(x) = x x^\dagger$, where $x^\dagger$ is the conjugate transposed of row vector $x$. 
When $x$ is associative we can define the \textit{projector} $[x] = x^\dagger x/ N(x)$, which is a primitive idempotent hermitian matrix in $\mathrm{Herm}(\rho, \mathbb{F})$. Projectors of this form belong to the projective space $\mathbb{FP}^{\rho-1}$, represented as primitive idempotent hermitian matrices, as described in \cite{hoggar_t-designs_1982}.
If $r$ in $\mathbb{F}^\rho$ is commutative or $\mathbb{F}$ is associative then we can define the following pair of \textit{reflection isometries} acting on $\mathbb{F}^\rho$ and $\mathbb{FP}^{\rho-1}$:
\begin{align*}
    \mathtt{W}_r: x \mapsto x (I_\rho - 2[r]), \quad 
     [x] \mapsto (I_\rho - 2[r]) [x] (I_\rho - 2[r]).
\end{align*}
That both actions of $\mathtt{W}_r$ are in fact respectively isometries of $\mathbb{F}^\rho$ and $\mathbb{FP}^{\rho-1}$ is verified in \cite{nasmith_octonions_2022}.
We can use these actions to construct spherical and projective $t$-designs, or simply designs. 
A \textit{spherical design} $X$ is a finite subset of points $X \subset \Omega_{\rho d}$ where $\Omega_{\rho d}$ is a sphere. A \textit{projective design} is a finite subset of points $X \subset \mathbb{FP}^{\rho-1}$ where $\mathbb{FP}^{\rho-1}$ is a projective space. A design $X$ has certain properties such as tightness, angle set $A$, and strength $t$ which are reviewed in \cite{nasmith_octonions_2022}. Of note, the only two strictly projective tight $5$-designs are shown to have a common construction in \cite{nasmith_octonions_2022}, where the \textit{common construction} is defined as follows:
\begin{definition}[Common Construction]
    \label{commonconstruction}
    Let $r_1, r_2, \ldots, r_n$ be commutative vectors in $\mathbb{F}^\rho$ (or let $\mathbb{F}$ be associative).
    Let $G$ and $H$ be the group of isometries generated by $\mathtt{W}_{r_1}, \mathtt{W}_{r_2}, \ldots, \mathtt{W}_{r_n}$ respectively acting on $\mathbb{F}^\rho$ and $\mathbb{FP}^{\rho-1}$. 
    If $G$ is finite then the orbit of $r_1, r_2, \ldots, r_n$ defines a spherical design on $\Omega_{\rho d}$, since $\mathbb{F}^\rho \cong \mathbb{R}^{\rho d}$, with a corresponding projective design on $\mathbb{RP}^{\rho d - 1}$. 
    If $H$ is finite then the orbit of $[r_1], [r_2], \ldots, [r_n]$ defines a projective design on $\mathbb{FP}^{\rho-1}$.
\end{definition}

Suitable vectors $r_1, \ldots, r_6$ exist such that the common construction above yields the two strictly projective tight $5$-designs \cite{nasmith_octonions_2022}. The coordinates given in \cite{nasmith_octonions_2022} ensure that the tight $5$-design on $\mathbb{RP}^{23}$ is given by the lines spanned by the short vectors of Wilson's octonion Leech lattice construction \cite{wilson_octonions_2009}.
In this paper we will instead work with octonion integers to describe the common construction of these two tight $5$-designs. We will also clarify Leech lattice constructions using octonion integer triples and describe the Suzuki chain subgroups of Leech lattice automorphisms as octonion reflection groups. 

\section{Octonion Integers}

This section reviews octonion algebra and octonion integer ring concepts in order to provide a complete description of the common construction of \cite{nasmith_octonions_2022}  in terms of octonion integers. Many important properties of octonions and their integer rings are described in \cite{springer_octonions_2000}, \cite{baez_octonions_2002} \cite{conway_quaternions_2003}. 

A \textit{composition algebra} is an algebra equipped with a nondegenerate quadratic form $N$ that satisfies the composition rule $N(xy) = N(x)N(y)$. A composition algebra is \textit{unital} if it also includes an identity element. A division composition algebra lacks isotropic vectors, i.e. $N(x) = 0$ only for $x = 0$. 
A major theorem due to Hurwitz confirms that there are precisely four unital division composition algebras $\mathbb{F}$ over the real numbers: the real field $\mathbb{R}$, the complex field $\mathbb{C}$, the quaternions $\mathbb{H}$, and the octonions $\mathbb{O}$. 
The octonion algebra $\mathbb{O}$ contains the others as subalgebras.

The corresponding \textit{inner product} $\langle x,y\rangle =  N(x+y) - N(x) - N(y)$ is twice the standard Euclidean inner product. 
This means that $N(x) = \frac{1}{2}\langle x,x\rangle$, and the standard Euclidean inner product is $\frac{1}{2}\langle x,y\rangle$.
\begin{remark}
    In our definition of inner product $\langle x,y\rangle$ we omit the usual factor of $\frac{1}{2}$ for convenience when working with octonion integers modulo $2$ in what follows. This also follows the convention in \cite{springer_octonions_2000}.
\end{remark}
The \textit{real component} of an octonion $x$ is the component projected onto the identity: $\mathrm{Re}(x) = \frac{1}{2}\langle 1, x\rangle$. The difference $\mathrm{Im}(x) = x - \mathrm{Re}(x)$ is the \textit{imaginary component} of $x$.
The \textit{octonion conjugate} is defined as $\overline{x} = 2\mathrm{Re}(x) - x$.
The octonion product ensures that $N(x) = x \overline{x} = \overline{x} x$ and that every octonion satisfies the following characteristic equation:
\begin{align*}
    x^2 - 2\mathrm{Re}(x) x + N(x) = 0.
\end{align*}

In what follows we will denote by $\mathbb{R}(x_1, x_2, \ldots, x_n)$ the subalgebra of $\mathbb{O}$ generated by the products and $\mathbb{R}$-linear combinations of octonions $1, x_1,x_2,\ldots, x_n$. 
Any octonion not contained in the real subalgebra, so that $x \ne \mathrm{Re}(x)$, generates a complex subalgebra $\mathbb{R}(x) \cong \mathbb{C}$. Any two octonions $x,y$ generating distinct complex subalgebras $\mathbb{R}(x) \ne \mathbb{R}(y)$ will generate a quaternion subalgebra $\mathbb{R}(x,y) \cong \mathbb{H}$. Finally, any three octonions $x,y,z$ that pairwise generate distinct quaternion subalgebras $\mathbb{R}(x,y)$, $\mathbb{R}(y,z)$, $\mathbb{R}(x,z)$ will generate the full octonion algebra $\mathbb{R}(x,y,z) = \mathbb{O}$. 
Every octonion belongs to some complex and quaternion subalgebra of $\mathbb{O}$. 
The complex numbers form a commutative and associative algebra, while the quaternion algebra is not commutative but still associative. 
The octonions are neither commutative nor associative but still have many special properties and symmetries.

An \textit{octonion algebra automorphism} is an invertible $\mathbb{R}$-linear map $\sigma: \mathbb{O} \rightarrow \mathbb{O}$ that also preserves the octonion product, so that $\sigma(xy) = \sigma(x)\sigma(y)$. The group $\mathrm{Aut}(\mathbb{O})$ of all such octonion algebra automorphisms is a Lie group of type $G_2$.
The group $\mathrm{Aut}(\mathbb{O})$ is transitive on \textit{imaginary units}, namely octonions $i$ such that $\mathrm{Re}(i) = 0$ and $N(i) = 1$.
It follows that $\mathrm{Aut}(\mathbb{O})$ is transitive on the complex subalgebras of $\mathbb{O}$, which each have the form $\mathbb{C} \cong \mathbb{R}(i)$ for some imaginary unit $i$. 
Likewise, $\mathrm{Aut}(\mathbb{O})$ is transitive on ordered pairs of orthogonal imaginary units, namely all $(i,j)$ such that $\langle i,j\rangle = 0$. 
This ensures that $\mathrm{Aut}(\mathbb{O})$ is transitive on quaternion subalgebras of $\mathbb{O}$ since these all have the form $\mathbb{H} \cong \mathbb{R}(i,j)$ for some pair of orthogonal imaginary units $(i,j)$. 
Finally, a \textit{basic triple} is an ordered triple of orthogonal imaginary units $(i,j,l)$, where $i,j,l$ pairwise generate distinct quaternion subalgebras.
The group $\mathrm{Aut}(\mathbb{O})$ is also transitive on basic triples, so we can select any basic triple $(i,j,l)$ in $\mathbb{O}$ without loss of generality.

\begin{remark}
    For any quaternion subalgebra $\mathbb{H}\subset\mathbb{O}$, the group $\mathrm{Aut}(\mathbb{O})$ contains an involution fixing $\mathbb{H}$ and multiplying the orthogonal component by $-1$. That is, for any basic triple $(i,j,l)$, the map $(i,j,l) \mapsto (i,j,-l)$ defines an octonion algebra automorphism fixing $\mathbb{H} = \mathbb{R}(i,j)$.
\end{remark}

As described in \cite{conway_quaternions_2003}, we now define the more familiar octonion \textit{standard basis} $\{i_t \mid t \in \mathrm{PL}(7) \}$, indexed by the projective line $\mathrm{PL}(7) = \{\infty\}\cup \mathbb{F}_7$, in terms of some basic triple $(i,j,l)$:
\begin{align*}
    i_\infty &= 1, & i_0 &= -(ij)l, & i_1 &= il, & i_2 &= i,\\
    i_3 &= j, & i_4 &= l, & i_5 &= ij, & i_6 &= jl.
\end{align*}
The units in this basis multiply as expected according to their properties in the quaternion subalgebras they pairwise define. 
That is, for all $t\ne \infty$, the set $\{1,i_t, i_{t+1}, i_{t+3}\}$ forms a standard quaternion basis with $i_{t}i_{t+1}i_{t+3} = -1$. 
The $\mathrm{PL}(7)$ indexing of the standard basis vectors exhibits certain helpful octonion algebra automorphisms, namely those defined on the basis by $i_t \mapsto i_{t+1}$ and $i_t \mapsto i_{2t}$ (here $1 = i_\infty$ is fixed and the remaining indices are computed modulo $7$). 

We will call $\mathbb{Z}$ the \textit{rational integers}, to distinguish from octonion integer rings in what follows. 
A \textit{Gravesian integer ring} is all $\mathbb{Z}$-linear combinations of some standard basis \cite[p. 100]{conway_quaternions_2003}. 
Equivalently, any basic triple $(i,j,l)$ defines a Gravesian integer ring $\mathbb{Z}(i,j,l)$, consisting of all $\mathbb{Z}$-linear combinations of $1, i, j, l$ and their products, which include the standard basis vectors.
Since a Gravesian integer ring contains a basic triple, and since $\mathrm{Aut}(\mathbb{O})$ is transitive on basic triples, the octonion automorphism group is transitive on Gravesian integer rings. 
This means we can select a representative and speak of \textit{the} Gravesian integer ring without loss of generality.

The Gravesian integers are an example of an octonion \textit{order}: a subring of the octonion algebra where each subring element $x$ has $\mathbb{Z}$-valued $2\mathrm{Re}(x)$ and $N(x)$ \cite[p. 100]{conway_quaternions_2003}.
If an octonion order contains basic triple $(i,j,l)$ it must also contain the corresponding standard basis $\{i_t\mid t \in \mathrm{PL}(7)\}$ defined above.
Therefore, any octonion order containing a basic triple also contains the corresponding Gravesian integer ring.
An \textit{octonion arithmetic} is a maximal order. 
It is known that there are precisely seven octonion arithmetics containing any given Gravesian integer ring, and that these form a single orbit under the octonion algebra automorphism subgroup generated by $i_t \mapsto i_{t+1}$ \cite[p. 100]{conway_quaternions_2003}. It follows that any octonion arithmetic containing a basic triple is one of seven containing the Gravesian integer ring defined by that basic triple.

We want to distinguish one of the seven octonion arithmetics containing the Gravesian integers as canonical. We will do so using properties of some underlying basic triple. 
Basic triple $(i,j,l)$ defines a unique pair of opposite imaginary units in the standard basis corresponding to all possible triple products of $i,j,l$, namely $\{i_0, - i_0\}$ (recall that $i_0 = -(ij)l = l(ij)$, etc.).
We will select a canonical octonion arithmetic for basic triple $(i,j,l)$ that distinguishes $i_0$ from the other units in the standard basis defined by that basic triple.
To do so, we note that there are three quaternion double bases in $\{\pm i_t \mid \mathrm{PL}(7)\}$ containing $\pm i_0$, namely $\{\pm 1, \pm i_0, \pm i_r, \pm i_{3r}\}$ for $r= 1,2,4$ (indices computed modulo $7$).
The rings $\mathbb{Z}(i_0, \omega_r)$ with $\omega_r = \frac{1}{2}\left(-1 + i_0 + i_r + i_{3r}\right)$, for $r$ one of $1,2,4$, are isomorphic to the \textit{Hurwitz integer ring} $\mathsf{H}$ \cite[p. 57]{conway_quaternions_2003}.  
The three Hurwitz integer rings $\mathbb{Z}(i_0, \omega_1)$, $\mathbb{Z}(i_0, \omega_2)$, $\mathbb{Z}(i_0, \omega_4)$ together generate one of the seven arithmetics containing the basic triple $(i,j,l)$.
Indeed, this arithmetic has the form $\mathbb{Z}(\omega_1, \omega_2, \omega_4)$. 
Accordingly, we define the \textit{canonical arithmetic} $\mathsf{O}$ containing basic triple $(i,j,l)$ as the ring,
\begin{align*}
    \mathsf{O} = \mathsf{O}_{i,j,l} = \mathbb{Z}(\omega_{1}, \omega_{2}, \omega_{4}), \quad 
    \omega_{r} = \frac{1}{2}\left(-1 + i_0 + i_r + i_{3r}\right).
\end{align*}
The ring $\mathsf{O}$ is called the \textit{octavian integers} in \cite[p. 99]{conway_quaternions_2003} but we will call it the \textit{octonion integer ring}.
The remaining six octonion arithmetics containing $(i,j,l)$ are given by the orbit of $\mathsf{O}$ under the automorphism defined by $i_t \mapsto i_{t + 1}$.

We now verify that any octonion arithmetic containing a basic triple is isomorphic to the canonical arithmetic defined above. 

\begin{lemma}
    The octonion algebra automorphism group is transitive on octonion arithmetics containing at least one basic triple.
\end{lemma}

\begin{proof}
    Let $\mathsf{O}$ be an octonion arithmetic containing a basic triple $(i,j,l)$. It follows that $\mathsf{O}$ contains the Gravesian integer ring defined by $(i,j,l)$ and is one of the seven isomorphic arithmetics containing that order.
    Likewise, let $\mathsf{O}'$ be an octonion arithmetic containing a basic triple $(i',j',l')$, one of seven isomorphic arithmetics containing this basic triple. 
    The image of $\mathsf{O}'$ under any octonion automorphism is also an isomorphic arithmetic containing the image of basic triple $(i', j', l')$. 
    The octonion algebra automorphism group is transitive on basic triples, which means that there exists an automorphism $\sigma$ such that $\sigma(i') = i$, $\sigma(j') = j$, $\sigma(l') = l$.
    Since $\sigma(\mathsf{O}')$ is also an arithmetic and contains $(i,j,l)$, it must be one of the seven arithmetics containing this basic triple. Therefore $\mathsf{O}$ and $\sigma(\mathsf{O}')$ are in the same orbit of the automorphism subgroup generated by cycle $i_t \mapsto i_{t+1}$.
    Therefore an automorphism exists mapping $\mathsf{O}'$ to $\mathsf{O}$.
\end{proof}

Using the \textit{double} Euclidean inner product $\langle x,y\rangle = N(x+y) - N(x) - N(y)$, the canonical arithmetic $\mathsf{O}$ of basic triple $(i,j,l)$ has the geometry of an $\mathtt{E}_8$ lattice. 
Since we have $N(x) = \frac{1}{2}\langle x,x\rangle$, the units of $\mathsf{O}$ correspond to $\mathtt{E}_8$ roots and we can select a simple root system as a basis. 

\begin{remark}
For the canonical arithmetic $\mathsf{O}$ defined above, the following Coxeter-Dynkin diagram provides a simple $\mathtt{E}_8$ root system:

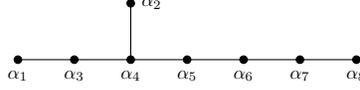
\begin{align*}
    \dynkin[edge length=.75cm, labels = {\alpha_1, \alpha_2, \alpha_3, \alpha_4, \alpha_5, \alpha_6, \alpha_7, \alpha_8}]{E}{8}
\end{align*}

\begin{align*}
    \alpha_1 &= \frac{1}{2}(-i_1+i_5+i_6+i_0), &
    \alpha_2 &= \frac{1}{2}(-i_1-i_2-i_4-i_0), \\
    \alpha_3 &= \frac{1}{2}(i_2+i_3-i_5-i_0), &
    \alpha_4 &= \frac{1}{2}(i_1-i_3+i_4+i_5), \\
    \alpha_5 &= \frac{1}{2}(-i_2+i_3-i_5+i_0), &
    \alpha_6 &= \frac{1}{2}(i_2-i_4+i_5-i_6), \\
    \alpha_7 &= \frac{1}{2}(-i_1-i_3+i_4-i_5), &
    \alpha_8 &= \frac{1}{2}(-1 + i_1-i_4+i_6). \\
\end{align*}
These particular simple roots have been chosen so that the highest root $\beta$ is opposite the identity,
\begin{align*}
    \beta = \dynkin[edge length=.75cm, labels*= {2,3,4,6,5,4,3,2}]{E}{8} = - 1,
\end{align*}
and also so that $\alpha_1,\ldots, \alpha_7$ are the simple roots of an $\mathtt{E}_7$ sublattice of purely imaginary octonion integers. This choice of simple roots is not uniquely defined by these properties. 
\end{remark}

The ring automorphism group $\mathrm{Aut}(\mathsf{O})$ has type $G_2(2) \cong \mathrm{PSU}_3(3):2$ and order $12096$. 
We can represent $\mathrm{Aut}(\mathsf{O})$ as a finite subgroup of $\mathrm{Aut}(\mathbb{O})$, the Lie group of automorphisms.

\begin{lemma}
    Every element of $\mathrm{Aut}(\mathsf{O})$ is a restriction of a unique element of $\mathrm{Aut}(\mathbb{O})$ to $\mathsf{O} \subset \mathbb{O}$.
\end{lemma}

\begin{proof}
    The group $\mathrm{Aut}(\mathsf{O})$ is known to be of type $G_2(2) \cong \mathrm{PSU}_3(3):2$, with cardinality $12096$ \cite[pp. 132-134]{wilson_finite_2009}. A ring automorphism group preserves units and we can represent $\mathrm{Aut}(\mathsf{O})$ faithfully using a permutation representation acting on the $240$ units of $\mathsf{O}$.
    Any map acting on $\mathbb{O}$ of the form $x \mapsto a^{-1} x a$ with $\mathrm{Re}(a^3) = a^3$ is an octonion algebra automorphism \cite[p. 98]{conway_quaternions_2003}. Computation using GAP on a canonical copy of $\mathsf{O}$ verifies that the $56$ units $\omega$ in $\mathsf{O}$ of order $3$ define permutations on the units of $\mathsf{O}$ of the form $x \mapsto \omega^{-1} x \omega$, and these permutations generate a ring automorphism group of order $6048$ and type $\mathrm{PSU}_3(3)$. 
    Suppose that $\mathsf{O}$ contains basic triple $(i,j,l)$ as described above. There exists an octonion algebra automorphism such that $(i,j,l) \mapsto (i,j,-l)$, since $(i,j,-l)$ is also a basic triple. 
    This octonion algebra automorphism fixes the subalgebra generated by $i$ and $j$ while negating the perpendicular component of $\mathbb{O}$. 
    Computation confirms that this algebra automorphism also permutes the units in $\mathsf{O}$ and is not contained in the $\mathrm{PSU}_3(3)$ ring automorphism subgroup given above. 
    Therefore the ring automorphism group $\mathrm{Aut}(\mathsf{O})$ has a permutation representation on the $240$ units of $\mathsf{O}$ and a matrix representation as endomorphisms of $\mathbb{O}$ in $\mathrm{Aut}(\mathbb{O})$. That is, we can characterize each element of $\mathrm{Aut}(\mathsf{O})$ as an automorphism in $\mathrm{Aut}(\mathbb{O})$ restricted to $\mathsf{O}$. 
    Finally, we show uniqueness. Let $\mathsf{O}$ be the canonical arithmetic of basic triple $(i,j,l)$. The orbit of ordered triple $(i,j,l)$ under the action of $\mathrm{Aut}(\mathsf{O})$ has length $12096$. 
    That is, the only ring automorphism in $\mathrm{Aut}(\mathsf{O})$ fixing $(i,j,l)$ is the identity. 
    Let $\sigma$ be an octonion automorphism fixing basic triple $(i,j,l)$. It follows that $\sigma$ also fixes the standard basis defined by this basic triple. Since $\sigma$ is an automorphism of an $\mathbb{R}$-algebra that fixes a basis it must be the identity automorphism.  
    It follows that the only algebra automorphism that restricts to the ring identity automorphism is the identity. Therefore, each ring automorphism is a restriction of a unique algebra automorphism. 
\end{proof}

\begin{remark}
    We can also describe $\mathrm{Aut}(\mathsf{O})$ as the stabilizer of $\mathsf{O}$ in $\mathrm{Aut}(\mathbb{O})$. 
    That is, we can take the subgroup of octonion automorphisms mapping basic triple $(i,j,l)$ to $(i',j',l')$ such that $\mathsf{O}_{i,j,l} = \mathsf{O}_{i',j',l'}$, i.e. such that the canonical arithmetics defined by the two triples are equivalent.
    This includes the maps $(i,j,l) \mapsto (\omega^{-1} i \omega, \omega^{-1} j \omega, \omega^{-1} l \omega)$ for any $\omega$ in $\mathsf{O}$ with $\omega^3 = 1$ and $(i,j,l) \mapsto (i,j,-l)$.
    Under the $\mathrm{Aut}(\mathsf{O})$ action generated by these maps, the basic triple $(i,j,l)$ belongs to an orbit of $12096$ basic triples for which $\mathsf{O}$ is the canonical arithmetic. 
\end{remark}

Octonion integers $\mathsf{O}$ form a non-associative ring and have the property that every ideal is a two-sided principle ideal of the form $n\mathsf{O}$, for $n$ a rational integer \cite[pp. 109-110]{conway_quaternions_2003}. 
The quotient ring $\mathsf{O}/2\mathsf{O}$ is a finite simple non-associative ring.
In fact, the ring $\mathsf{O}/2\mathsf{O}$ is precisely the unique finite octonion algebra over the field  with two elements $\mathbb{F}_2$ \cite[pp. 19-22]{springer_octonions_2000}.
The automorphism group of the ring $\mathsf{O}/2\mathsf{O}$ is the automorphism group $G_2(2)$ of this ring as an $\mathbb{F}_2$-algebra. 

The residue classes of $\mathsf{O}/2\mathsf{O}$ (of the form $x + 2\mathsf{O}$) each have representatives $x$ of minimal norm either $0$, $1$, or $2$.
With respect to the \textit{standard Euclidean inner product} $\frac{1}{2}\langle x,y\rangle$, $\mathsf{O}$ is isomorphic to the scaled $\mathtt{E}_8/\sqrt{2}$ lattice, with $240$ norm $1$ and $2160$ norm $2$ elements:
\begin{align*}
    |\mathsf{O}/2\mathsf{O}| = 
    1 + \frac{240}{2} + \frac{2160}{16}.
\end{align*}
Following \cite[p. 136]{conway_quaternions_2003} we will call the $16$ norm $2$ representatives of a residue class a \textit{frame}.
Each frame is geometrically eight orthogonal pairs of opposite norm $2$ vectors.
The norm $1$ and $2$ elements of $\mathsf{O}$ form the following orbits under the action of $\mathrm{Aut}(\mathsf{O})$:
    \begin{enumerate}
        \item The zero of $x^2 -2x + 1$, the identity $1$.
        \item The zero of $x^2 +2x + 1$, the element $-1$.
        \item The $126$ zeros of $x^2 + 1$, order $4$ imaginary units denoted $i$.
        \item The $56$ zeros of $x^2 + x + 1$, order $3$ units denoted $\omega$.
        \item The $56$ zeros of $x^2 - x + 1$, order $6$ units $-\omega$.
        \item The $126$ zeros of $x^2 -2 x + 2$, norm $2$ elements $1+i$.
        \item The $126$ zeros of $x^2 + 2 x + 2$, norm $2$ elements $-1-i$.
        \item The $576$ zeros of $x^2 - x + 2$, norm $2$ elements denoted $\lambda$.
        \item The $576$ zeros of $x^2 + x + 2$, norm $2$ elements $-\lambda$.
        \item The $756$ zeros of $x^2 + 2$, norm $2$ elements denoted $i + j$.
    \end{enumerate}

To summarize, when working with the octonion integers we can use the arithmetic $\mathsf{O}$ without loss of generality since an arithmetic containing a standard basis (indeed a basic triple) is unique up to automorphism. Furthermore, we can select a root $\lambda$ of $x^2 + x + 2$ in $\mathsf{O}$ without loss of generality since any choice is unique up to automorphism.
This will allow us to construct octonion Leech lattices using properties of octonion integers but without reference to a particular choice of coordinates.
In what follows we will also make use of the finite ring $\mathsf{O}/2\mathsf{O}$ to simplify computations and proofs. 

\section{Octonion Integer Leech Lattices}

In what follows we will make use of certain theorems described in \cite{lepowsky_e8-approach_1982} to identify $\mathtt{E}_8$ sublattices of $\mathsf{O}$ and Leech sublattices of $\mathsf{O}^3$. 
First we describe a method to identify the $\mathtt{E}_8$ and Leech lattices using the classification of unimodular lattices in low dimensions.

\begin{lemma}
\label{latticeidentification}
    The Gosset lattice $\mathtt{E}_8$ is the unique unimodular lattice in $\mathbb{R}^8$ with minimal norm at least $2$.
    The Leech lattice $\Lambda_{24}$ is the unique unimodular lattice in $\mathbb{R}^{24}$ with minimal norm at least $4$. 
\end{lemma}

\begin{proof}
    Every unimodular lattice is either odd (type I) or even (type II).
    A unimodular lattice with a norm $1$ vector is of the form $\mathbb{Z} \oplus L$ for some other unimodular lattice $L$. 
    The only unimodular lattice in $\mathbb{R}^8$, with minimal norm $2$, is the Gosset lattice.
    So any unimodular lattice in $\mathbb{R}^8$ without vectors of norm $1$ has minimal norm at least $2$ and must be the Gosset lattice.
    The unimodular lattices in dimension $d \le 24$ with minimal norm at least $2$ are classified in \cite[chaps. 16-18]{conway_sphere_2013}.
    There are $24$ even and $156$ odd unimodular lattices in $\mathbb{R}^{24}$ with minimal norm at least $2$. 
    The only odd unimodular lattice in $\mathbb{R}^{24}$ with minimal norm at least $3$ is the \textit{odd Leech lattice} $O_{24}$, which contains vectors with minimal norm $3$.
    The only even unimodular lattice in $\mathbb{R}^{24}$ with minimal norm at least $3$ is the Leech lattice $\Lambda_{24}$, which contains vectors with minimal norm $4$. 
    This classification confirms that any unimodular lattice (whether odd or even) in $\mathbb{R}^{24}$ with minimal norm at least $4$ is the Leech lattice. 
\end{proof}

In order to understand the main result of \cite{lepowsky_e8-approach_1982} we need to describe the concept of a totally isotropic subspace of quotient $L/2L$ for $L$ an even unimodular lattice.
Let $L/ 2L$ be a $\mathbb{F}_2$-vectorspace defined on the residue classes modulo $2L$ (i.e., on the additive cosets $x+ 2L$). The corresponding inner product on $L/ 2L$ has value $0$ or $1$ in $\mathbb{F}_2$ according to whether any given representatives in $L$ have even or odd Euclidean inner product. 
A \textit{totally isotropic subspace} of $L/ 2L$ is the image, modulo $2L$, of a sublattice of $L$ which contains $2L$ and for which all inner products are even. 
Since $L$ is even unimodular, $\mathrm{dim}~L/ 2L$ is also even. A \textit{maximal totally isotropic subspace} $M/2L$ of $L/ 2L$ is a totally isotropic subspace with dimension $\mathrm{dim}~M/2L = \frac{1}{2}\mathrm{dim}~L/ 2L$. 
We are interested in the sublattice preimage $M \subset L$ of maximal totally isotropic subspace $M/2L \subset L/ 2L$. 

\begin{theorem}
    \cite{lepowsky_e8-approach_1982} Let $L$ be an even unimodular lattice with $M$ a sublattice satisfying,
    \begin{align*}
        2 L \subset M \subset L.
    \end{align*}
    The lattice $M/\sqrt{2}$ is an even unimodular lattice if and only if $M/2L$ is a maximal totally isotropic subspace of $L/2L$.
    \label{lepowskytheorem}
\end{theorem}

As described in \cite{lepowsky_e8-approach_1982}, the requirement that $M/2L$ is totally isotropic ensures that $M/\sqrt{2}$ is an even lattice and the condition that $M/2L$ is \textit{maximal} as a totally isotropic subspace ensures that $M/\sqrt{2}$ is also unimodular. 
We verify the maximal condition by simply checking that $\mathrm{dim}~M/2L = \frac{1}{2} \mathrm{dim}~L/2L$. 

We will now adapt this theorem to the case where $L = \mathsf{O}^n$. To do so, we extend the octonion norm $N$ to octonion vectors $x$ in $\mathsf{O}^n$ such that $N(x) = x x^\dagger$. 
The double Euclidean inner product on octonion vectors is, 
\begin{align*}
    \langle x, y\rangle = N(x+y) - N(x) - N(y) = x y^\dagger + y x^\dagger. 
\end{align*}
For any pair of residue classes $x+2\mathsf{O}^n$, $y+2\mathsf{O}^n$ in the $\mathbb{F}_2$-vectorspace $\mathsf{O}^n/2\mathsf{O}^n$, we define the inner product as $\langle x,y\rangle \mod 2$.
The following theorem adapts Theorem \ref{lepowskytheorem} so that it identifies even unimodular sublattices of $\mathsf{O}^n$ with respect to $\frac{1}{2}\langle x,y\rangle$:
\begin{theorem}
    Let $\mathsf{N}$ be a sublattice of $\mathsf{O}^n$ that satisfies,
    \begin{align*}
        2 \mathsf{O}^n \subset \mathsf{N} \subset \mathsf{O}^n.
    \end{align*}
    The lattice $\mathsf{N}$ is an even unimodular lattice with respect to Euclidean inner product $\frac{1}{2}\langle x,y\rangle$ if and only if $\mathsf{N}/2\mathsf{O}^n$ is a maximal totally isotropic subspace of $\mathsf{O}^n /2\mathsf{O}^n$ with respect to inner product $\langle x,y\rangle \mod 2$.
    \label{evenunimodularsublatticesofOn}
\end{theorem}

\begin{proof}
    With respect to inner product $\langle x,y\rangle$, the lattice $\mathsf{O}^n$ is an $\mathtt{E}_8^n$ lattice. 
    It is therefore an even unimodular lattice relative to $\langle x,y\rangle$ and we can apply Theorem \ref{lepowskytheorem} to the sublattice $\mathsf{N}$. 
    According to that theorem, $\mathsf{N}/\sqrt{2}$ is an even unimodular lattice with respect to $\langle x,y \rangle$ if and only if $\mathsf{N}/2\mathsf{O}^n$ is a maximal totally isotropic subspace of $\mathsf{O}^n /2\mathsf{O}^n$.
    The condition that $\mathsf{N}/\sqrt{2}$ is an even unimodular lattice with respect to $\langle x,y \rangle$ is equivalent to the condition that $\mathsf{N}$ is an even unimodular lattice with respect to $\frac{1}{2}\langle x,y\rangle$. 
    The condition that $\mathsf{N}/2\mathsf{O}^n$ is a totally isotropic subspace is equivalent to the condition that for any $x,y$ in $\mathsf{N}$, inner product $\langle x,y\rangle$ is an even integer. 
    Therefore, this theorem is a special case of Theorem \ref{lepowskytheorem}.
\end{proof}

\begin{example}
    \label{e8sublattices}
    Let $n=1$. Then $\mathsf{O}/2\mathsf{O}$ is an $8$-dimensional $\mathbb{F}_2$-vectorspace. 
    This vectorspace contains a zero vector, $120$ norm $1$ vectors and $135$ non-zero isotropic vectors (with $N(x) = 0$). 
    A totally isotropic subspace is spanned by isotropic vectors satisfying $\langle x,y\rangle \mod 2 =0$ for any $x \ne y$ in the subspace.
    If we construct a graph on the $135$ isotropic vectors, assigning an edge when $\langle x,y\rangle \mod 2 =0$, we obtain a strongly regular graph $\mathrm{srg}( 135, 70, 37, 35 )$. 
    This graph has $270$ maximal cliques, each corresponding to the non-zero vectors of a four-dimensional totally isotropic subspace of $\mathsf{O}/2\mathsf{O}$. 
    Since these totally isotropic subspaces have dimension equal to half the dimension of $\mathsf{O}/2\mathsf{O}$, they are maximal totally isotropic subspaces. These subspaces also form a single orbit. Each clique defines a maximal totally isotropic subspace $\mathsf{N}/2\mathsf{O}$, and the corresponding pre-image $\mathsf{N} \subset\mathsf{O}$ is an even unimodular lattice with respect to Euclidean inner product $\frac{1}{2}\langle x,y\rangle$. Each clique has sixteen vectors, forming a finite $\mathbb{F}_2$-vectorspace of dimension $4$. 
    Since even unimodular lattice $\mathsf{N}$ has minimal norm $2$, it must be the $\mathtt{E}_8$ lattice. Therefore, there are precisely $270$ $\mathtt{E}_8$ sublattices of $\mathsf{O}$.
\end{example}

\begin{remark}
    The fact that $\mathtt{E}_8$ contains $270$ sublattices isometric to $\sqrt{2}\mathtt{E}_8$ is discussed in a mathematics blog post \cite{baez_integral_2014-1}. 
    The use of $\mathsf{O}/2\mathsf{O}$ structure and Theorem \ref{evenunimodularsublatticesofOn} makes it possible to determine this fact by examining the properties of a strongly regular graph on $135$ points, the isotropic vectors of $\mathsf{O}/2\mathsf{O}$. 
\end{remark}

We now describe some further properties of the $\mathtt{E}_8$ sublattices  of $\mathsf{O}$. In the following lemmas, let $s,s'$ denote norm $2$ elements in $\mathsf{O}$ and 
let $\lambda$ denote any zero of $x^2 + x + 2$ in $\mathsf{O}$, i.e. $\mathrm{Re}(\lambda) = -\frac{1}{2}$ and $N(\lambda) = 2$.

\begin{lemma}
    \label{OssOlemma}
    Two $\mathtt{E}_8$ sublattices of $\mathsf{O}$ of the form $\mathsf{O}s$, $\mathsf{O}s'$ are equal if and only if $s \equiv s' \mod 2\mathsf{O}$.
    The same is true of sublattices $s\mathsf{O}$ and $s'\mathsf{O}$.
\end{lemma}

\begin{proof}
    Left or right octonion multiplication is a conformal mapping, meaning that the $\langle sx,sy\rangle = N(s)\langle x,y\rangle = \langle x s,y s\rangle$ \cite[p. 5]{springer_octonions_2000}.
    Since $N(s) = 2$, this ensures that $\mathsf{O}s$ and $s\mathsf{O}$ are both $\mathtt{E}_8$ sublattices of $\mathsf{O}$, relative to inner product $\frac{1}{2}\langle x,y\rangle$. 
    Two $\mathtt{E}_8$ sublattices of $\mathsf{O}$ are equivalent subsets of $\mathsf{O}$ if and only if they have the same image modulo $2\mathsf{O}$. 
    But since mapping modulo $2\mathsf{O}$ is a ring homomorphism, $\mathsf{O}s$ and $\mathsf{O}s'$ have the same image if and only if $s \equiv s' \mod 2\mathsf{O}$. The same argument applies for $s\mathsf{O}$ and $s'\mathsf{O}$.
\end{proof}

\begin{remark}
    Lemma \ref{OssOlemma} ensures that each frame in $\mathsf{O}$ defines a pair of $\mathtt{E}_8$ sublattices of the form $\mathsf{O}s$ and $s\mathsf{O}$ for $s$ any norm $2$ frame representative. 
\end{remark}

\begin{lemma}
    Each $\mathtt{E}_8$ sublattice of $\mathsf{O}$ has the form of either $\mathsf{O}s$ or $s \mathsf{O}$.
\end{lemma}

\begin{proof}
    We can construct $135$ lattices of the form $\mathsf{O}s$ and $135$ more lattices of the form $s \mathsf{O}$, yielding a total of $270$ $\mathtt{E}_8$ sublattices of $\mathsf{O}$. By taking the images modulo $2\mathsf{O}$ we can verify that there are no repetitions, so that the preimages are $135+135$ distinct sublattices of $\mathsf{O}$.
    By Example \ref{e8sublattices} there are only $270$ $\mathtt{E}_8$ sublattices of $\mathsf{O}$ so we have found them all. 
\end{proof}

\begin{lemma}
    \label{nonintersecting}
    Let $s \not\equiv s' \mod 2\mathsf{O}$.
    Then $\mathsf{O}s \cap s'\mathsf{O} \ne 2\mathsf{O}$. Also, $\mathsf{O}s \cap \mathsf{O}s' = 2\mathsf{O}$ if and only if $N(s+s')$ is odd. 
\end{lemma}

\begin{proof}
    We can verify this lemma efficiently by computation using the ring $\mathsf{O}/2\mathsf{O}$. 
    That is, $s \not\equiv s' \mod 2\mathsf{O}$ ensures that $\mathsf{O}s$ and $\mathsf{O}s'$ have distinct images modulo $2\mathsf{O}$ and therefore represent distinct lattices (likewise for left multiplication by $s,s'$).
    Since the images of $\mathsf{O}s$ and $s\mathsf{O}$ modulo $2\mathsf{O}$ are $\mathbb{F}_2$-vectorspaces, they contain the residue class $0+2\mathsf{O}$, which implies that $2\mathsf{O} \subset \mathsf{O}s$ and $2\mathsf{O} \subset s\mathsf{O}$.  
    The properties of the images modulo $2\mathsf{O}$ determine the properties of the intersections given above and can be quickly verified by computation.
\end{proof}

\begin{corrollary}
     Since $N(\lambda + \overline{\lambda}) = 1$ we have $\mathsf{O}\lambda \cap \mathsf{O}\overline{\lambda} = 2\mathsf{O}$.
\end{corrollary}

The following can be verified by computation using a canonical copy of $\mathsf{O}$ and is also proven in \cite{conway_quaternions_2003}. 

\begin{lemma}
    \cite[pp. 138-141]{conway_quaternions_2003} The stabilizer in $\mathrm{Aut}(\mathsf{O})$ of residue class $\lambda + 2\mathsf{O}$ is a subgroup of type $\mathrm{PSL}_2(7)$, which also stabilizes $\overline{\lambda} + 2 \mathsf{O}$.
\end{lemma}

We can now describe how to construct a Leech sublattice of octonion integer triples in $\mathsf{O}^3$, adapting the techniques described in \cite{lepowsky_e8-approach_1982} for octonion integers. 
We begin by selecting any two $\mathtt{E}_8$ sublattices of $\mathsf{O}$ with respect to $\frac{1}{2}\langle x,y\rangle$, call them $\Phi$ and $\Psi$, that satisfy the following requirements:
\begin{align*}
    \Phi + \Psi = \mathsf{O}, \quad \Phi\cap \Psi = 2\mathsf{O}.
\end{align*}
By Theorem \ref{evenunimodularsublatticesofOn}, the images of $\Phi$ and $\Psi$ modulo $2\mathsf{O}$ are both maximal totally isotropic subspaces of $\mathsf{O}/2\mathsf{O}$ that only intersect in zero, the $0+ 2\mathsf{O}$ residue class. 

\begin{definition}
    \label{dualE8leech}
    Let $\Phi$ and $\Psi$ be any two $\mathtt{E}_8$ sublattices of $\mathsf{O}$, with respect to $\frac{1}{2}\langle x,y\rangle$, that satisfy $\Phi + \Psi = \mathsf{O}$ and $\Phi\cap \Psi = 2\mathsf{O}$. 
    Let $\Lambda(\Phi, \Psi)$ be the following sublattice of $\mathsf{O}^3$:
    \begin{align*}
        \Lambda(\Phi,\Psi) = \left\lbrace (a,b,c) \subset \mathsf{O}^3 \mid a+b, b+c, a+c \in \Phi, a+b+c \in \Psi \right\rbrace
    \end{align*}
    Equivalently,
    \begin{align*}
        \Lambda(\Phi, \Psi) = \left\lbrace (x_1 + z, x_2 + z, x_3 + z) \mid  x_i \in \Phi, x_1 + x_2 + x_3, z \in \Psi \right\rbrace.
    \end{align*}
    and also,
    \begin{align*}
        \Lambda(\Phi, \Psi) = \left\lbrace (x + y + z, x + z, y + z) \mid  x,y \in \Phi, z \in \Psi \right\rbrace.
    \end{align*}
\end{definition}

\begin{remark}
    To verify that the three definitions of $\Lambda(\Phi,\Psi)$ are equivalent, it suffices to verify that the second definition is a sublattice of the first, the third a sublattice of the second, and the first a sublattice of the third. 
\end{remark}

\begin{theorem}
    The lattice $\Lambda = \Lambda(\Phi, \Psi)$ of Definition \ref{dualE8leech}, with inner product $\frac{1}{2}\langle x, y\rangle$, is a Leech lattice 
\end{theorem}

\begin{proof}
    We need to apply Theorem \ref{evenunimodularsublatticesofOn} to show that $\Lambda$ is an even unimodular lattice with respect to $\frac{1}{2}\langle x, y\rangle$. The following sketch parallels a proof given in \cite{lepowsky_e8-approach_1982}, which does not involve octonions.
    Since $\Phi$ and $\Psi$ are sublattices of $\mathsf{O}$ we know that $\Lambda \subset \mathsf{O}^3$.
    Since $2\mathsf{O}$ is in both $\Phi$ and $\Psi$, we know that $\Lambda$ contains $(2a,0,0)$ for any $a$ in $\mathsf{O}$. The same is true for the other coordinate positions. 
    Therefore we also have $2\mathsf{O}^3 \subset \Lambda$. 
    Modulo $2\mathsf{O}$, $\Lambda$ consists of three orthogonal totally isotropic subspaces relative to inner product $\langle x, y\rangle \mod 2$. Each of these subspaces has dimension $4$. These are the three subspaces spanned respectively by the vectors of the form $(x,x,0)$, $(y,0,y)$, and $(z,z,z)$. 
    Therefore $\Lambda / 2\mathsf{O}^3$ has dimension $12$, which is maximal in $\mathsf{O}^3/2 \mathsf{O}^3$ since it is half of $24$. 
    Therefore it is a maximal totally isotropic subspace of $\mathsf{O}^3/2\mathsf{O}^3$. 
    This ensures that $\Lambda$ is an even unimodular lattice relative to the standard Euclidean inner product $\frac{1}{2}\langle x,y\rangle$. 
    Finally, since $\Lambda$ lacks any norm $2$ vectors, by Theorem \ref{latticeidentification} it is a Leech lattice.
\end{proof}

We can describe a family of octonion Leech lattices with respect to certain norm $2$ octonions as follows.
\begin{definition}
    \label{munulattice}
    Let $s,s'$ in $\mathsf{O}$ be norm $2$ octonion integers  with $\mathsf{O}s \cap \mathsf{O}s' = 2 \mathsf{O}$.
    We define $\Lambda(s,s')$ to be the Leech lattice $\Lambda(\mathsf{O}s, \mathsf{O}s')$.
\end{definition}

\begin{remark}
    In what follows we restrict to describing Leech lattices of the form $\Lambda(\mathsf{O}s, \mathsf{O}s')$, rather than $\Lambda(s\mathsf{O}, s'\mathsf{O})$, since the latter lattices can easily be obtained from the former by octonion conjugation.
\end{remark}

\begin{lemma}
    The lattices $\Lambda(s,s')$ and $\Lambda(t,t')$ are equal if and only if $s \equiv t \mod 2\mathsf{O}$ and $s' \equiv t' \mod 2\mathsf{O}$.
\end{lemma}

\begin{proof}
    This follows from Lemma \ref{OssOlemma} and the fact that $\Lambda(s,s') = \Lambda(\mathsf{O}s, \mathsf{O}s')$ and $\Lambda(t,t') = \Lambda(\mathsf{O}t, \mathsf{O}t')$ are equal if and only if $\mathsf{O}s = \mathsf{O}t$ and $\mathsf{O}s' = \mathsf{O}t'$.    
\end{proof}

\begin{remark}
    The octonion integer units $\alpha_1,\ldots, \alpha_8$ given above form a basis for $\mathsf{O}$, since they correspond to simple roots. This means that the lattice $\Lambda(s,s')$ has the following basis,
    \begin{align*}
        (\alpha_i s, \alpha_i s, 0), ~ (0, \alpha_is, \alpha_is), ~ (\alpha_i s', \alpha_i s', \alpha_i s'), \quad i=1,2,\ldots, 8.
    \end{align*}
\end{remark}

\begin{definition}
    The following \textit{translation} or \textit{multiplication maps} acting on $\mathsf{O}$ are defined as in \cite{conway_quaternions_2003}:
    \begin{align*}
        \mathtt{L}_x: y \mapsto xy, \quad 
        \mathtt{R}_x: y \mapsto yx, \quad 
        \mathtt{B}_x: y \mapsto xyx.
    \end{align*}
\end{definition}

We can extend these translations to a diagonal action on $\mathsf{O}^3$ as follows.

\begin{definition}
    For $\mathtt{X}$ one of $\mathtt{L}, \mathtt{R}, \mathtt{B}$ we define the following diagonal action on $(x,y,z) \subset \mathsf{O}^3$:
    \begin{align*}
        \mathtt{X}_u(x,y,z) = (\mathtt{X}_u(x), \mathtt{X}_u(y), \mathtt{X}_u(z)). 
    \end{align*}
\end{definition}

For any unit $u$ in $\mathsf{O}$, the translation maps $\mathtt{L}_u$, $\mathtt{R}_u$, and $\mathtt{B}_u$ are isometries of both $\mathsf{O}$ and $\mathsf{O}^3$.
This means that given Leech lattice $\Lambda(s,s')$ we can obtain another Leech lattice using a translation map $\mathtt{X}_u$. In fact, we can use octonion translations to describe a single orbit of $8640$ octonion integer Leech lattices of the form $\Lambda(s,s')$.
In order to describe this orbit, we can use certain helpful properties of translation maps due to properties of the octonion algebra as a Moufang loop. 

\begin{lemma}
    For any $x,y, u$ in $\mathsf{O}$ with $u$ a unit, the translations given above satisfy the following: 
    \begin{align*}
        \mathtt{L}_u(xy) &= \mathtt{B}_u(x)\mathtt{L}_{\overline{u}}(y), & 
        \mathtt{R}_u(xy) &= \mathtt{R}_{\overline{u}}(x)\mathtt{B}_{u}(y), &
        \mathtt{B}_u(xy) &= \mathtt{L}_u(x)\mathtt{R}_u(y).
    \end{align*}
\end{lemma}

\begin{proof}
    These are another form of the Moufang identities that the octonions satisfy, as described in \cite[74]{conway_quaternions_2003}:
     \begin{align*}
        u(xy) = (u x u) (u^{-1} y), \quad 
        (xy)u = (x u^{-1})(uy u), \quad 
        u(xy) u = (ux)(y u).
    \end{align*}
\end{proof}

These properties of translations allow us to permute $\mathtt{E}_8$ sublattices of $\mathsf{O}$, which all have the form $\mathsf{O}s$ or $s\mathsf{O}$ for some norm $2$ octonion integer $s$. 

\begin{lemma}
    Let $u$ be a unit and $s$ a norm $2$ element in $\mathsf{O}$. Then we have,
    \begin{align*}
        u(\mathsf{O}s) &= \mathsf{O}(\overline{u}s), & 
        (\mathsf{O}s) u &= \mathsf{O}(us u), & 
        u(\mathsf{O}s)u &= \mathsf{O}(s u), \\
        u(s\mathsf{O}) &= (us u)\mathsf{O}, & 
        (s\mathsf{O}) u &= (s \overline{u})\mathsf{O}, & 
        u(s\mathsf{O})u &= (us)\mathsf{O}.
    \end{align*}
    which can also be written as,
    \begin{align*}
         \mathtt{L}_u(\mathsf{O}s) &= 
             \mathsf{O}
            \mathtt{L}_{\overline{u}}(s), & 
        \mathtt{R}_u(\mathsf{O}s) &= 
            \mathsf{O}
            \mathtt{B}_{u}(s), &
        \mathtt{B}_u(\mathsf{O}s) &=     
            \mathsf{O}
            \mathtt{R}_u(s), \\
        \mathtt{L}_u(s \mathsf{O}) &= 
            \mathtt{B}_u(s)
            \mathsf{O}, & 
        \mathtt{R}_u(s \mathsf{O}) &= 
            \mathtt{R}_{\overline{u}}(s)
            \mathsf{O}, &
        \mathtt{B}_u(s \mathsf{O}) &= 
            \mathtt{L}_u(s)
            \mathsf{O}.
    \end{align*}
\end{lemma}

\begin{proof}
    These follow from the Moufang identities, the fact that $\overline{u}$ is a unit for $u$ a unit in $\mathsf{O}$, and the fact that any translation of $\mathsf{O}$ by a unit simply permutes the elements of $\mathsf{O}$.
\end{proof}

\begin{theorem}
    Let $u$ be a norm 1 element and let $s, s'$ be norm 2 elements in $\mathsf{O}$. 
    Let $s + s'$ have odd norm. 
    Then the $2\cdot \mathrm{SO}_8^+(2)$ isometry group, generated by translations $\mathtt{X}_u$ for all units $u$, permutes the lattices of Definition \ref{munulattice} as follows:
    \begin{align*}
        \mathtt{L}_u\Lambda(s,s') &= \Lambda(\mathtt{L}_{\overline{u}}(s),\mathtt{L}_{\overline{u}}(s')), \\
        \mathtt{R}_u \Lambda(s,s') &= \Lambda(\mathtt{B}_u(s),\mathtt{B}_u(s')), \\
        \mathtt{B}_u \Lambda(s,s') &= \Lambda(\mathtt{R}_u(s),\mathtt{R}_u(s')).
    \end{align*}
\end{theorem}

\begin{proof}
    The identities here follow from the definition of $\Lambda(s, s') = \Lambda(\mathsf{O}s, \mathsf{O}s')$ and the Moufang identities expressed in terms of $\mathtt{L}$, $\mathtt{R}$, $\mathtt{B}$.
    The $2\cdot \mathrm{SO}_8^+(2)$ group of translations $\mathtt{X}_u$ can be constructed and identified in GAP using a canonical copy of $\mathsf{O}$. 
\end{proof}

\begin{theorem}
    Let $s, s'$ in $\mathsf{O}$ be norm $2$ octonion integers with $s + s'$ having an odd norm. 
    Then there exists a norm $2$ root of $x^2 + x +2$ in $\mathsf{O}$, called $\lambda$, and unit $u$ in $\mathsf{O}$ such that,
    \begin{align*}
        \Lambda(s,s') = \mathtt{L}_u \Lambda(\overline{\lambda},\lambda).
    \end{align*}
    There are $8640 = 72 \cdot 120$ choices of $\lambda$ and $u$, taken modulo $2\mathsf{O}$, corresponding to the $8640$ lattices of the form $\Lambda(s,s')$.
    These form a single orbit under the action of the $2 \cdot \mathrm{SO}_8^+(2)$ generated by $\mathtt{X}_u$ for $u$ any octonion unit and $\mathtt{X} = \mathtt{L}, \mathtt{R}, \mathtt{B}$.
    \label{leechulambda}
\end{theorem}

\begin{proof}
    First, we construct a graph on the $135$ norm $2$ octonion integers modulo $2\mathsf{O}$ where two vertices are adjacent when the sum of their representatives has odd norm. This graph is a $\mathrm{srg}(135, 64, 28, 32)$ and each \textit{directed edge} $s \rightarrow s'$ represents a Leech lattice of the form $\Lambda(s,s')$ (in general $\Lambda(s,s') \ne \Lambda(s',s)$). 
    Second, we can construct an edge-transitive automorphism group of this graph generated by the translation maps $\mathtt{X}_u$ acting on vertices. Indeed, this automorphism group is transitive on directed edges. 
    Third, we can recover all $8640$ directed edges by computing $(\mathtt{L}_{\overline{u}}(\overline{\lambda}), \mathtt{L}_{\overline{u}}(\lambda)) \mod 2\mathsf{O}$ for the $72$ choices of $\lambda$ and the $120$ choices of $u$ modulo $2\mathsf{O}$. 
    These facts can be confirmed by computation using GAP. 
\end{proof}

\begin{remark}
    The $2 \times 8640$ Leech lattices of the form $\Lambda(\mathsf{O}s, \mathsf{O}s')$ or $\Lambda(s\mathsf{O}, s'\mathsf{O})$ are identified and described in the mathematics blog posts \cite{baez_integral_2014} and \cite{baez_integral_2014-1}, by John Baez and Greg Egan. They use a combination of direct computation and analysis of $\mathtt{E}_8$ lattice properties to identify these lattices. 
    The approach shown here instead makes use of the ring homomorphism $\mathsf{O} \mapsto \mathsf{O}/2\mathsf{O}$ and Theorem \ref{lepowskytheorem} of \cite{lepowsky_e8-approach_1982} to simplify the calculations to properties of a strongly regular graph on $135$ points, which can easily be explored by computation or other graph theory techniques. 
\end{remark}

\section{Leech Lattice Symmetries}

By construction the lattice $\Lambda(\overline{\lambda},\lambda)$ is symmetric under all coordinate permutations and coordinate sign changes, a $2\times S_4$ group action. We will call these the \textit{coordinate automorphisms} of $\Lambda(\overline{\lambda},\lambda)$.
The lattice $\Lambda(\overline{\lambda},\lambda)$ is also fixed under the octonion automorphisms that preserve our canonical choice of arithmetic $\mathsf{O}$ and residue class $\lambda + 2\mathsf{O}$.
We call automorphisms belonging to this $\mathrm{PSL}_2(7) \subset \mathrm{Aut}(\mathsf{O})$ the \textit{scalar automorphisms} of $\Lambda(\overline{\lambda},\lambda)$.

We now introduce \textit{reflection automorphisms} of $\Lambda(\overline{\lambda},\lambda)$. 
A computer search of a canonical copy of  $\Lambda(\overline{\lambda}, \lambda)$ shows that the short vectors of this Leech lattice contain $2\times 1260$ commutative octonion integer triples, namely triples where the three coefficients generate a commutative subalgebra of $\mathbb{O}$. 
Of these triples, $2\times 1176$ define reflections $\mathtt{W}_r$ acting on $\mathsf{O}^3$ that preserve the Leech lattice $\Lambda(\overline{\lambda},\lambda)$. 
These include $720$ of the form $(2u,0,0)$ for $u$ a unit in $\mathsf{O}$ and $1440$ of the form $(s,\pm s,0)$ for $s$ any $\mathtt{E}_8$ root in $\mathsf{O}\overline{\lambda}$.
For $r$ a commutative vector in $\Lambda(\overline{\lambda},\lambda)$ of the form $(2u,0,0)$ or $(s, \pm s, 0)$, the corresponding reflection $\mathtt{W}_r$ is a coordinate symmetry in $2 \times S_4$.
The remaining $2 \times 96$ commutative short vectors in $\Lambda(\overline{\lambda},\lambda)$ that define reflection symmetries have the form $(1, 1, \lambda')$ for $\lambda' \equiv \lambda \mod 2\mathsf{O}$, under all coordinate permutations and sign changes. 

\begin{figure}
    \begin{centering}
        \begin{tabular}{ccc}
        Type & Number & Comment  \\
        \hline \hline
        $(2u, 0, 0)$  & $3 \times 240$ & for $u$ a unit in $\mathsf{O}$\\
        $(s,\pm s, 0)$ & $6 \times 240$ &  for $s$ a root in $\mathsf{O}\overline{\lambda}$\\
        $(\pm 1, \pm 1, \lambda')$  & $12 \times 16$ & for $\lambda' \equiv \lambda \mod 2\mathsf{O}$\\
        \hline
    \end{tabular}    
    \end{centering}
    \caption{Commutative vectors $r$ in $\Lambda(\overline{\lambda},\lambda)$ defining reflection symmetries $\mathtt{W}_r$ of $\Lambda(\overline{\lambda},\lambda)$.}
    \label{fig:my_label}
\end{figure}

\begin{theorem}
    \label{refltheorem}
    Let $\lambda$ be an octonion integer in $\mathsf{O}$ and zero of $x^2 + x + 2$.
    The vector $r = (1,1,\lambda)$, for any $\lambda \equiv \lambda' \mod 2\mathsf{O}$, defines an automorphism $\mathtt{W}_r$ of the Leech lattice $\Lambda(\overline{\lambda'},\lambda')$. 
\end{theorem}

\begin{proof}
    To begin, we write $\Lambda(\overline{\lambda'},\lambda') = \Lambda(\overline{\lambda},\lambda)$ since $\lambda' \equiv \lambda \mod 2\mathsf{O}$. Note that $\lambda + \overline{\lambda} = -1$ and that $1-\overline{\lambda} = - \lambda^2$. 
    A vector $(a,b,c) \in \Lambda(\overline{\lambda},\lambda)$ is reflected as follows:
    \begin{align*}
        \mathtt{W}_r(a,b,c) = (a,b,c) - 2(a,b,c)[r].
    \end{align*}
    Since $(a,b,c)$ is in the Leech lattice, we will have $\mathtt{W}_r(a,b,c)$ also in the Leech lattice when  $2(a,b,c)[r]$ is in the Leech lattice.
    Let $(a',b',c') = 2(a,b,c)[r]$. 
    We can compute $(a',b',c')$ by first computing $[r]$.
    \begin{align*}
        [r] = \frac{r^\dagger r}{r r^\dagger} = \frac{1}{4}\left(\begin{matrix} 1 \\ 1 \\ \overline{\lambda} \end{matrix}\right) \left(1,1,\lambda\right)
        = \frac{1}{4}\left(
        \begin{matrix}
        1 & 1 & \lambda \\
        1 & 1 & \lambda \\
        \overline{\lambda} & \overline{\lambda} & 2 
        \end{matrix}\right).
    \end{align*}
    This means that we have,
    \begin{align*}
        (a',b',c') &= 2(a,b,c)[r] = \frac{1}{2}(a,b,c)\left(
        \begin{matrix}
        1 & 1 & \lambda \\
        1 & 1 & \lambda \\
        \overline{\lambda} & \overline{\lambda} & 2 
        \end{matrix}\right) \\
        &= \frac{1}{2}\left(a + b + c\overline{\lambda}, a + b + c\overline{\lambda}, a\lambda + b \lambda + 2c \right).
    \end{align*}
    We first verify that $a'+ b' + c' \in \mathsf{O} \lambda$.
    \begin{align*}
        a'+b'+c' = (a+b)\left(1 + \frac{\lambda}{2}\right) + c (\overline{\lambda} +1).
    \end{align*}
    By construction $a+b$ is in $\mathsf{O} \overline{\lambda}$ so we have $a+b = \alpha \overline{\lambda}$ for some $\alpha$ in $\mathsf{O}$. Furthermore, $\alpha, \overline{\lambda}, 1 + \frac{\lambda}{2}$ belong to a common associative subalgebra of $\mathbb{O}$. This means that we use $\overline{\lambda}\lambda = 2$ and simplify as follows:
    \begin{align*}
        a'+b'+c' = (\alpha + c )(\overline{\lambda} +1).
    \end{align*}
    But $\overline{\lambda} + 1 = - \lambda$ so $a' + b' + c'$ is in $\mathsf{O} \lambda$. 
    Next we show that $a' + b' \in \mathsf{O} \overline{\lambda}$.
    \begin{align*}
        a'+b' = a + b + c\overline{\lambda}.
    \end{align*}
    By construction $a+b = \alpha \overline{\lambda}$ for some $\alpha$ in $\mathsf{O}$. This ensures that $a'+b'$ is in $\mathsf{O}\overline{\lambda}$.
    Finally we check that $b'+c' = a'+c' \in \mathsf{O} \overline{\lambda}$.
    \begin{align*}
        a'+c' = \frac{1}{2}(a+b)\left(1 + \lambda\right)+ \frac{1}{2}c\left(2 + \overline{\lambda} \right)
        = - \frac{1}{2}(a+b)\overline{\lambda} - \frac{1}{2} c \overline{\lambda}^2
    \end{align*}
    We know that $(a+b) = \alpha \overline{\lambda}$:
    \begin{align*}
        a'+c' = - \frac{1}{2}(\alpha +c )\overline{\lambda}^2.
    \end{align*}
    We also know that $a+b+c = \beta \lambda$ and can write $c = \beta \lambda - \alpha \overline{\lambda}$:
    \begin{align*}
        a'+c' &= - \frac{1}{2}(\alpha  - \alpha \overline{\lambda} )\overline{\lambda}^2 - \frac{1}{2}(\beta \lambda )\overline{\lambda}^2 \\
        &= - \frac{1}{2}\alpha(1  - \overline{\lambda} )\overline{\lambda}^2 - \beta\overline{\lambda}.
    \end{align*}
    We use the properties $1-\overline{\lambda} = -\lambda^2$ and $\lambda^2 \overline{\lambda}^2 = 4$:
    \begin{align*}
        a'+c' &= \frac{1}{2}\alpha \lambda^2 \overline{\lambda}^2 - \beta\overline{\lambda} \\
        &= 2\alpha - \beta\overline{\lambda}.
    \end{align*}
    Since $2\alpha$ is in $2\mathsf{O} \subset \mathsf{O}\overline{\lambda}$, this ensures that $a'+c' = b'+c'$ are in $\mathsf{O}\overline{\lambda}$.
\end{proof}

Theorem \ref{refltheorem} ensures that $\mathtt{W}_r$ for $r = (1,1, \lambda')$ and $\lambda' \equiv \lambda \mod 2\mathsf{O}$ is a reflection symmetry of Leech lattice $\Lambda(\overline{\lambda},\lambda)$. 
It turns out that that each of the $2\times 96$ short vectors in $\Lambda(\overline{\lambda},\lambda)$ of the form $(\pm 1, \pm 1, \lambda')$ for $\lambda' \equiv \lambda \mod 2\mathsf{O}$ defines a reflection symmetry of $\Lambda(\overline{\lambda},\lambda)$.
Computation in GAP on a canonical example verifies the following two theorems about how to generate the full automorphism group of the Leech lattice $\Lambda(\overline{\lambda},\lambda)$. 

\begin{theorem}
    The automorphism group $2\cdot \mathrm{Co}_1$ of Leech lattice $\Lambda(\overline{\lambda},\lambda)$ is generated by reflections $\mathtt{W}_r$ for $r = (\pm 1, \pm 1,\lambda')$ under all coordinate permutations and with $\lambda' \equiv \lambda \mod 2\mathsf{O}$. 
    \label{Co1gensA}
\end{theorem}

\begin{theorem}
    The automorphism group $2\cdot \mathrm{Co}_1$ of Leech lattice $\Lambda(\overline{\lambda},\lambda)$ is generated by the $2 \times S_4$ coordinate permutations and sign changes together with reflections $\mathtt{W}_r$ for $r = (1, 1,\lambda')$ for $\lambda' \equiv \lambda \mod 2\mathsf{O}$ and $\mathrm{Re}(\lambda) = \mathrm{Re}(\lambda')$. 
    \label{Co1gensB}
\end{theorem}

From these facts it follows that all automorphisms of $\Lambda(\overline{\lambda},\lambda)$ are compositions of octonion reflection automorphisms. Although the Leech lattice does not contain any real reflection symmetries, the octonion Leech lattice $\Lambda(\overline{\lambda},\lambda)$ has octonion reflection symmetries that generate the full automorphism group. 

We can also describe automorphisms of any octonion Leech lattice of the form $\Lambda(s,s')$. 
Since by Theorem \ref{leechulambda} we know that $\Lambda(s,s') = \mathtt{L}_u \Lambda(\overline{\lambda},\lambda)$ we also have $\mathtt{L}_{\overline{u}}\Lambda(s,s') = \Lambda(\overline{\lambda},\lambda)$. Let $r = (1,1,\lambda')$ for any $\lambda' \equiv \lambda \mod 2\mathsf{O}$, under any coordinate permutation and sign change.
Then $\mathtt{W}_r$ acting on $\mathbb{O}^3$ is an automorphism of $\Lambda(\overline{\lambda}, \lambda)$ and $\mathtt{L}_u \mathtt{W}_r \mathtt{L}_{\overline{u}}$ is an automorphism of $\mathtt{L}_u \Lambda(\overline{\lambda},\lambda)$. We can generate the automorphism group $2\cdot \mathrm{Co}_1$ of Leech lattice $\mathtt{L}_{u}\Lambda(\overline{\lambda},\lambda)$ using involutions of the form $\mathtt{L}_u \mathtt{W}_r \mathtt{L}_{\overline{u}}$. 
In the special case where $u = 1$, the involution $\mathtt{L}_u \mathtt{W}_r \mathtt{L}_{\overline{u}}$ becomes an octonion reflection.

\section{The Common Construction and Octonion Integer Triples}

We now use the common construction of \cite{nasmith_octonions_2022} to describe the correspondence between certain structures in the octonion integer Leech lattice and generalized hexagons in the octonion projective plane. 
We will also describe an approach to generating the Suzuki chain of Leech lattice automorphism subgroups using octonion reflections. 

First introduced by Jacques Tits, a \textit{finite generalized $n$-gon} of order $(s,t)$ is a block design on $v$ points such that (1) each block contains $s+1$ points, (2) each point belongs to $t+1$ blocks, and (3) the point-block incidence graph has diameter $n$ and girth $2n$ (described also as a bipartite graph with diameter $n$ and girth $2n$ in \cite[5.6]{godsil_algebraic_2001}).
We will write $\mathrm{Gh}(s,t)$ to denote the \textit{generalized hexagon} ($n = 6$) of order $(s,t)$. 
An ordinary hexagon, taken as a block design with edges for blocks and vertices for points, is a $\mathrm{Gh}(1,1)$. 
Cohen demonstrates in \cite{cohen_exceptional_1983} that the generalized hexagons $\mathrm{Gh}(2,1)$, $\mathrm{Gh}(2,2)$, and $\mathrm{Gh}(2,8)$ can be constructed as Jordan frames (i.e., three orthogonal primitive idempotents) in the octonion projective plane. 
We will show that the common construction of \cite{nasmith_octonions_2022} puts these generalized hexagon structures in correspondence with certain integral octonion triples with reflections that generate Suzuki chain subgroups of Leech lattice automorphisms.

Second, the automorphism group of the Leech lattice is the Conway group $\mathrm{Co}_0 = 2 \cdot \mathrm{Co}_1$. The tight $5$-design in $\mathbb{RP}^{23}$ corresponding to the $98280$ lines spanned by the Leech lattice short vectors has the sporadic simple group $\mathrm{Co}_1$ for an isometry group. 
The Suzuki chain is a series of subgroups of $\mathrm{Co}_1$ constructed in the following manner \cite[p. 219]{wilson_finite_2009}. 
The group $\mathrm{Co}_1$ has a maximal subgroup $A_9 \times S_3$. The symmetric group $S_3$ centralizes $A_9$ in $\mathrm{Co}_1$. The chain of subgroups $A_9 > A_8 > A_7 > \cdots > A_4 > A_3$ has the following corresponding chain of centralizers in $\mathrm{Co}_1$, known as the \textbf{Suzuki chain}:
\begin{align*}
    S_3 < S_4 < \mathrm{PSL}_2(7) < \mathrm{PSU}_3(3) < \mathrm{HJ} < G_2(4) < 3 \cdot \mathrm{Suz}.
\end{align*}
Here $\mathrm{HJ}$ is the Hall-Janko sporadic simple group and $\mathrm{Suz}$ is the Suzuki sporadic simple group. 
In what follows we will see how the group $\mathrm{Co}_1$ and the Suzuki chain subgroups can be constructed as octonion reflection groups, acting projectively on $\mathbb{R}^{24}/\{\pm 1\} \cong \mathbb{O}^3 /\{\pm 1\}$, generated by reflections on $\mathbb{O}^3$ of the form $\mathtt{W}_r$ for suitable choices of octonion triple $r$. 

\begin{remark}
    Consider the following related sequence of groups:
    \begin{align*}
        \mathrm{PSU}_3(3) < \mathrm{HJ} < G_2(4) < \mathrm{Suz}.
    \end{align*} 
    The groups in this sequence have permutation representations respectively of degrees $63$, $100$, $416$, and $1782$.
    The $\mathrm{Suz}$ group action on $1782$ points has rank $3$, and the point stabilizer is the group $G_2(4)$ with orbits of lengths $1 + 416 + 1365$. The action of $G_2(4)$ on the $416$ points also has rank $3$, and the point stabilizer is the group $\mathrm{HJ}$ with orbits of lengths $1+100+315$. 
    The action of $\mathrm{HJ}$ on the $100$ points also has rank $3$, and the point stabilizer is the group $\mathrm{PSU}_3(3)$ with orbits of lengths $1+36+63$.
    The action of $\mathrm{PSU}_3(3)$ on both the $36$ or $63$ length orbits is rank $4$, so the sequence of rank $3$ permutation groups stops (contrary to \cite[p. 124]{griess_twelve_1998}). As mentioned in \cite[p. 124]{griess_twelve_1998}, there is no rank $3$ permutation group containing $\mathrm{Suz}$ as a point stabilizer.
\end{remark}

Having described generalized hexagons and the Suzuki chain of Leech lattice symmetries, we now define certain octonion integer triples for use in the common construction in order to link the two concepts. 
Consider the following vectors with $\lambda$ some zero of $x^2 + x + 2$ in $\mathsf{O}$:
\begin{align*}
    S_\lambda = \left\{ (2, 0, 0), (\overline{\lambda}, 0, \overline{\lambda}), (1,1,\lambda) \right\}.
\end{align*}
Since the coefficients belong to a common complex subalgebra $\mathbb{C}$, we have $S_\lambda$ in $\mathbb{C}^3 \subset \mathbb{O}^3$ and the projectors $\{[x] \mid x \in S_\lambda\}$ in the complex projective subplane $\mathbb{CP}^2 \subset \mathbb{OP}^2$.

\begin{example}
    The common construction of \cite{nasmith_octonions_2022} applied to $\{r_1, r_2, r_3\} = S_\lambda$ yields $G = 2 \times \mathrm{PSL}_2(7)$ acting on $\mathbb{C}^3$, with a $42$ point design on $\Omega_6$. 
    The common construction also yields $H = \mathrm{PSL}_2(7)$ acting on $\mathbb{CP}^2$ and the orbit of the projectors of $S_\lambda$ define a $21$ point design. 
    The $21$ {points} of the design in $\mathbb{CP}^{2}$, and the blocks of three mutually orthogonal points, form a $\mathrm{Gh}(2,1)$ structure. Here we have a link between the Suzuki chain group $H = \mathrm{PSL}_2(7) \cong G / \{\pm 1\}$ and a generalized hexagon $\mathrm{Gh}(2,1)$ of Jordan frames on $\mathbb{CP}^2$.
    \label{single}
\end{example}

\begin{remark}
    The $2$-design in $\mathbb{CP}^2$ of Example \ref{single} is equivalent to Example 12 of \cite{hoggar_t-designs_1982}. 
    However, this design is mislabeled in \cite{hoggar_t-designs_1982} as a $3$-design at the special bound. A calculation confirms that it is neither a 3-design nor at the special bound.
\end{remark}

\begin{remark}
    The group $G = 2 \times \mathrm{PSL}_2(7)$ acting on $\mathbb{C}^3$ of Example \ref{single} contains a real reflection subgroup $2 \times S_4$. Indeed, the $42$ vectors generated by $S_\lambda$ under reflection $\mathtt{W}_r$ are precisely the orbits of the vectors in $S_\lambda$ under all coordinate permutations and sign changes, which is the action of $2 \times S_4$. 
    The projectors generating these reflections are the nine contained in the real projective subplane $\mathbb{RP}^2 \subset \mathbb{CP}^2$, which form a $1$-design. 
    Their reflection action on $\mathbb{O}^3$ generate the $2 \times S_4$ coordinate symmetries of $\Lambda(\overline{\lambda},\lambda)$ discussed above.
\end{remark}

We have seen that the stabilizer of the residue class $\lambda + 2 \mathsf{O}$ in $\mathsf{O}$ (a subgroup of $\mathrm{Aut}(\mathsf{O})$) is a group of type $\mathrm{PSL}_2(7)$.
We may call this group a \textit{frame stabilizer} since it stabilizes the frame containing $\lambda$ in $\mathrm{Aut}(\mathsf{O})$.
Under the action of this $\mathrm{PSL}_2(7)$ frame stabilizer, the vector $\lambda$ has an orbit of length $8$. 
The corresponding scalar action on the set of triples $S_\lambda$ also has length $8$, as does the scalar action on the $\mathrm{Gh}(2,1)$ structures generated by reflections $\mathtt{W}_r$ for $r$ in $S_\lambda$ vectors. 
Each $S_{\lambda'}$ in this orbit is defined by $\lambda' \equiv \lambda \mod 2\mathsf{O}$ and $\mathrm{Re}(\lambda') = \mathrm{Re}(\lambda)$. 

The permutation action of $\mathrm{PSL}_2(7)$ on the eight $\lambda' \equiv \lambda \mod 2\mathsf{O}$ with $\mathrm{Re}(\lambda) = \mathrm{Re}(\lambda')$ is $2$-transitive. However, this action is transitive on \textit{unordered} subsets $\{\lambda_i, \lambda_j, \lambda_k,\ldots\}$ of size $n$, except when $n=4$.
That is, the frame stabilizer is also transitive on pairs $\{\lambda_i, {\lambda_j}\}$ and triples $\{{\lambda_i}, {\lambda_j}, {\lambda_k}\}$. 
It is therefore also transitive on subsets of cardinality $5,6,7,8$. 
In contrast, there are three orbits of quadruples $\{{\lambda_i}, {\lambda_j}, {\lambda_k}, {\lambda_l}\}$, with lengths $14 + 42 + 14$. 
Both length $14$ orbits define Steiner systems $S(3,4,8)$ on the eight $\lambda'$ with $\lambda' \equiv \lambda \mod 2\mathsf{O}$ and $\mathrm{Re}(\lambda') = \mathrm{Re}(\lambda)$.

\begin{figure}
    \begin{align*}
    \begin{array}{cccc}
    \{r_1, \ldots, r_n\}     & G/\{\pm 1\} \subset \mathrm{Co}_1 & H \subset F_4 & \text{Example}\\
    \hline \hline 
    S_{\lambda_i} & \mathrm{PSL}_2(7) & \mathrm{PSL}_2(7) & \ref{single}
    \\
    S_{\lambda_i} \cup S_{\lambda_j} & \mathrm{PSU}_3(3) & \mathrm{PSU}_3(3) & \ref{double}
    \\
    S_{\lambda_i} \cup S_{\lambda_j} \cup S_{\lambda_k} & \mathrm{HJ} & ^3 D_4(2) & \ref{triple}
    \\
    S_{\lambda_i} \cup S_{\lambda_j} \cup S_{\lambda_k} \cup S_{\lambda_l} & G_2(4) & ^3 D_4(2) & \ref{quadrupleA}
    \\
    S_{\lambda_i} \cup S_{\lambda_j} \cup S_{\lambda_k} \cup S_{\lambda_l} & G_2(4) & & \ref{quadrupleB}
    \\
    S_{\lambda_i} \cup S_{\lambda_j} \cup S_{\lambda_k} \cup S_{\lambda_l} \cup S_{\lambda_m} & 3\cdot \mathrm{Suz} & & \ref{quintuple} 
    \\
    S_{\lambda_i} \cup S_{\lambda_j} \cup S_{\lambda_k} \cup S_{\lambda_l} \cup S_{\lambda_m} \cup S_{\lambda_n} & \mathrm{Co}_1 & & \ref{sextuple}\\
    \hline
    \end{array}
\end{align*}
    \caption{The common construction of \cite{nasmith_octonions_2022} applied to orbits of combinations of $S_\lambda$ under $\mathrm{PSL}_2(7) \subset \mathrm{Aut}(\mathsf{O})$. }
    \label{resultstable}
\end{figure}

We can now describe the Suzuki chain groups and their correspondence to certain generalized hexagons via the common construction of Definition \ref{commonconstruction}. 
In each case, the result is obtained by computation in GAP on a representative example.

\begin{example}
    If we apply the common construction to any pair $S_{\lambda_i} \cup S_{\lambda_j}$ we obtain $G/\{\pm 1\} = \mathrm{PSU}_3(3)$ and $H = \mathrm{PSU}_3(3)$. The corresponding design on $\mathbb{HP}^2 \subset \mathbb{OP}^2$ defines a $\mathrm{Gh}(2,2)$ finite geometry.
    \label{double}
\end{example}

\begin{example}
    If we instead apply the common construction to any triple $S_{\lambda_i} \cup S_{\lambda_j} \cup S_{\lambda_k}$ we obtain $G/\{\pm 1\} = \mathrm{HJ}$ and $H$ is $^3D_4(2)$. The corresponding design on $\mathbb{OP}^2$ defines a $\mathrm{Gh}(2,8)$ finite geometry. 
    \label{triple}
\end{example}

\begin{example}
One of the two Steiner systems on quadruples $\{{\lambda_i}, {\lambda_j}, {\lambda_k}, {\lambda_l}\}$ has the special property that $\{r_1, \ldots, r_n\} = S_{\lambda_i} \cup S_{\lambda_j} \cup S_{\lambda_k} \cup S_{\lambda_l}$ yields the two strictly projective tight $5$-designs under the common construction of \cite{nasmith_octonions_2022}. 
That is, for one of the two length $14$ orbits of quadruples,the common construction applied to the vectors of $S_{\lambda_i} \cup S_{\lambda_j} \cup S_{\lambda_k} \cup S_{\lambda_l}$ yields $G/\{\pm 1\} = G_2(4)$ and $H$ is $^3D_4(2)$. The corresponding spherical design on $\Omega_{24}$ is the same as that given by the short vectors of Leech lattice $\Lambda(\overline{\lambda},\lambda)$, with the corresponding tight projective $5$-design in $\mathbb{RP}^{23}$. 
The corresponding design on $\mathbb{OP}^2$ defines a $\mathrm{Gh}(2,8)$ finite geometry. 
This is the same finite geometry and $^3 D_4(2)$ group obtained by the common construction using any three $S_{\lambda_i}, S_{\lambda_j}, S_{\lambda_k}$ contained in the quadruple.
\label{quadrupleA}
\end{example}

\begin{example}
    The quadruples $\{{\lambda_i}, {\lambda_j}, {\lambda_k}, {\lambda_l}\}$ of the other length $14$ orbit and of the length $42$ orbit no longer yield the $\mathrm{Gh}(2,8)$ finite geometry on $\mathbb{OP}^2$ under the common construction, although they still yield $G/\{\pm 1\} = G_2(4)$ and the Leech lattice short vectors on $\mathbb{O}^3$. It is unclear whether the projectors of the generating set have a finite orbit under the action of corresponding group $H$, or whether the group $H$ is finite. 
    \label{quadrupleB}
\end{example}

\begin{example}
Any quintuple $S_{\lambda_i} \cup \ldots \cup S_{\lambda_j}$ yields $G/\{\pm 1\} = 3 \cdot \mathrm{Suz}$ and the same spherical design on $\Omega_{24}$ as its $G_2(4)$ subgroups do. 
\label{quintuple}
\end{example}

\begin{example}
Any sextuple $S_{\lambda_i} \cup \ldots \cup S_{\lambda_j}$ yields $G/\{\pm 1\} = \mathrm{Co}_1$ and the same spherical design on $\Omega_{24}$ as its $3 \cdot \mathrm{Suz}$ subgroups do. The same is true of a septuple or the union of the full set of eight $S_\lambda$. 
\label{sextuple}
\end{example}

The results described above are summarized in Fig. \ref{resultstable}. The Suzuki chain subgroups satisfy the following theorem.

\begin{theorem}
    The Suzuki chain subgroups of $\mathrm{Co}_1$,
    \begin{align*}
        \mathrm{PSL}_2(7) < \mathrm{PSU}_3(3) < \mathrm{HJ} < G_2(4) < 3\cdot \mathrm{Suz},
    \end{align*}
    are the quotients modulo $\{\pm 1\}$ of octonion reflection groups generated by reflections $\mathtt{W}_r$ acting on $\mathbb{O}^3$ with $r$ in the union of any subset of $\{S_{\lambda_i}, \cdots, S_{\lambda_j}\}$ of cardinality respectively $1,2,3,4,5$.
    \label{Suzukichain}
\end{theorem}

\begin{proof}
    The proof involves checking representative examples in GAP. \nocite{noauthor_gap_2020}
\end{proof}

\section{Conclusion}

We have seen that the Leech lattice can be constructed as a sublattice of octonion integer triples, determined up to unit $u$ and root $\lambda$ of $x^2 + x + 2$ taken modulo $2\mathsf{O}$, as described in Theorem \ref{leechulambda}.
We have also seen that the automorphism group of the Leech lattices with $u=1$ are generated by octonion reflections $\mathtt{W}_r$, as described in Theorems \ref{Co1gensA} and \ref{Co1gensB}.
For $u \ne 1$ these generating reflections are replaced by generating involutions $ \mathtt{L}_u \mathtt{W}_r \mathtt{L}_{\overline{u}} $. 
We have also seen that this Leech lattice construction (for $u = 1$) provides a simple means of generating the two strictly projective tight 5-designs, given in Example \ref{quadrupleA}. Specifically, the generating reflections used in the common construction satisfy a $S(3,4,8)$ Steiner system structure determined via $\mathrm{PSL}_2(7)$ octonion integer automorphisms stabilizing $\lambda + 2 \mathsf{O}$. 
Finally, we see that the Suzuki chain subgroups of $\mathrm{Co}_1$ have a simple octonion reflection construction given by Theorem \ref{Suzukichain}. The following remarks describe two remaining open questions.

\begin{remark}
    Using the notation given above, the element $\mathtt{W}_r$ for $r = (1,1,\lambda')$ is an involution in $2\cdot \mathrm{Co}_1$. What is the conjugacy class of this involution? Likely the full conjugacy class contains involutions that are not octonion reflections.
    Of note, an octonion reflection may have determinant $-1$ as an octonion matrix even though the $24 \times 24$ real matrix for the endomorphism acting on $\mathsf{O}^3$ as a real $24$-vector has determinant $1$.
\end{remark}

\begin{remark}
    An open question is whether there is some algebraic property of $\{\lambda_i, \lambda_j, \lambda_k, \lambda_l\}$ in the $S(3,4,8)$ of Example \ref{quadrupleA} that yields the two tight projective $5$-designs that distinguishes it from the other length $14$ orbit (which is also an $S(3,4,8)$) in Example \ref{quadrupleB}. 
\end{remark}

\bibliographystyle{amsalpha}
\bibliography{Bibliography}

\providecommand{\bysame}{\leavevmode\hbox to3em{\hrulefill}\thinspace}
\providecommand{\MR}{\relax\ifhmode\unskip\space\fi MR }
\providecommand{\MRhref}[2]{%
  \href{http://www.ams.org/mathscinet-getitem?mr=#1}{#2}
}
\providecommand{\href}[2]{#2}
\begin{thebibliography}{Wil09b}

\bibitem[Bae02]{baez_octonions_2002}
John Baez, \emph{The octonions}, Bulletin of the American Mathematical Society
  \textbf{39} (2002), no.~2, 145--205.

\bibitem[Bae14]{baez_integral_2014}
\bysame, \emph{Integral {Octonions} (part 11)}, December 2014.

\bibitem[BE14]{baez_integral_2014-1}
John Baez and Greg Egan, \emph{Integral {Octonions} ({Part} 9)}, November 2014.

\bibitem[Coh83]{cohen_exceptional_1983}
Arjeh~M. Cohen, \emph{Exceptional presentations of three generalized hexagons
  of order 2}, Journal of Combinatorial Theory, Series A \textbf{35} (1983),
  no.~1, 79--88 (en).

\bibitem[CS03]{conway_quaternions_2003}
John~H. Conway and Derek~A. Smith, \emph{On {Quaternions} and {Octonions}:
  {Their} {Geometry}, {Arithmetic}, and {Symmetry}}, CRC Press, 2003.

\bibitem[CS13]{conway_sphere_2013}
John~Horton Conway and Neil James~Alexander Sloane, \emph{Sphere packings,
  lattices and groups}, vol. 290, Springer Science \& Business Media, 2013.

\bibitem[GR01]{godsil_algebraic_2001}
Chris Godsil and Gordon~F. Royle, \emph{Algebraic {Graph} {Theory}}, Graduate
  {Texts} in {Mathematics}, Springer-Verlag, New York, 2001.

\bibitem[Gri98]{griess_twelve_1998}
Robert L.~Jr Griess, \emph{Twelve {Sporadic} {Groups}}, Springer {Monographs}
  in {Mathematics}, Springer-Verlag, Berlin Heidelberg, 1998 (en).

\bibitem[Hog82]{hoggar_t-designs_1982}
Stuart~G. Hoggar, \emph{t-{Designs} in projective spaces}, European Journal of
  Combinatorics \textbf{3} (1982), no.~3, 233--254.

\bibitem[LM82]{lepowsky_e8-approach_1982}
James Lepowsky and Arne Meurman, \emph{An {E8}-approach to the {Leech} lattice
  and the {Conway} group}, Journal of Algebra \textbf{77} (1982), no.~2,
  484--504 (en).

\bibitem[Nas22]{nasmith_octonions_2022}
Benjamin Nasmith, \emph{Octonions and the two strictly projective tight
  5-designs}, Algebraic Combinatorics \textbf{5} (2022), no.~3, 401--411.

\bibitem[noa20]{noauthor_gap_2020}
\emph{{GAP} -- {Groups}, {Algorithms}, and {Programming}, {Version} 4.11.0},
  2020.

\bibitem[SV00]{springer_octonions_2000}
Tonny~A. Springer and Ferdinand~D. Veldkamp, \emph{Octonions, {Jordan}
  {Algebras} and {Exceptional} {Groups}}, Springer {Monographs} in
  {Mathematics}, Springer-Verlag, Berlin Heidelberg, 2000 (en).

\bibitem[Wil09a]{wilson_finite_2009}
Robert~A. Wilson, \emph{The {Finite} {Simple} {Groups}}, Graduate {Texts} in
  {Mathematics}, Springer-Verlag, London, 2009 (en).

\bibitem[Wil09b]{wilson_octonions_2009}
\bysame, \emph{Octonions and the {Leech} lattice}, Journal of Algebra
  \textbf{322} (2009), no.~6, 2186--2190 (en).

\bibitem[Wil11]{wilson_conways_2011}
\bysame, \emph{Conway's group and octonions}, Journal of Group Theory
  \textbf{14} (2011), no.~1, 1--8 (en).

\end{thebibliography}

\end{document}